\documentclass[letterpaper, 10 pt, conference]{ieeeconf}  

\IEEEoverridecommandlockouts                              

\usepackage{graphics} 
\usepackage{epsfig} 
\usepackage{mathptmx} 
\usepackage{times} 
\usepackage{amsmath} 
 
\usepackage{amsthm} 
\usepackage{amssymb}  
\usepackage{mathrsfs}
\theoremstyle{plain}
\newtheorem{theorem}{Theorem}

\newtheorem{lemma}{Lemma}
\newtheorem{prop}{Proposition}
\theoremstyle{definition}
\newtheorem{definition}{Definition}
\theoremstyle{remark}

\newcommand{\noi}{\noindent}

\newcommand{\R}{\mathbb{R}}

\newcommand{\C}{\mathcal{C}}

\renewcommand{\epsilon}{\varepsilon}

\newcommand{\p}{\partial}
\newcommand{\eg}{\emph{e.g.}}

\title{\LARGE \bf
Sufficient conditions for time optimality of\\systems with control on the disk}

\author{Jean-Baptiste Caillau$^{1}$ and \underline{Micha\"el Orieux}$^{2}$%
\thanks{*This work was supported by the FMJH Program PGMO and EDF-Thales-Orange
(PGMO grant no. 2016-1753H).}
\thanks{$^{1}$Jean-Baptiste Caillau is with
        Universit\'e C\^ote d'Azur, CNRS, Inria, LJAD, France
        {\tt\small jean-baptiste.caillau@univ-cotedazur.fr}}%
\thanks{$^{2}$Micha\"el Orieux is with
	  SISSA, Via Bonomea 265, I-34136 Trieste
        {\tt\small morieux@sissa.it}}%
}

\begin{document}

\maketitle
\thispagestyle{empty}
\pagestyle{empty}

\begin{abstract}
The case of time minimization for affine control systems with control on the disk is
studied. After recalling the standard sufficient conditions for local optimality in
the smooth case, the analysis focusses on the specific type of singularities
encountered when the control is prescribed to the disk.
Using a suitable stratification, the regularity of the flow is analyzed, which helps to
devise verifiable sufficient conditions in terms of left and right limits of Jacobi
fields at a switching point. Under the appropriate assumptions, piecewise regularity
of the field of extremals is obtained.
\end{abstract}

\section{INTRODUCTION}
\noi In this paper we are interested in the minimum time control of some affine
control problems - namely, in dimension four for a rank two distribution. We deal with
sufficient conditions for optimality of extremal trajectories of such systems.
	This topic is a very active field of research, and a variety of different approaches exist 
	and have been applied to a large number of problems. 
	Geometric methods hold an important
	place in that regard.
	When the extremal flow is smooth, the 
	theory of conjugate points can be applied, and local optimality holds before the
	 first conjugate time. We recall this result below.
    The points where the extremal ceases to be globally optimal are cut points,
    and it is an extremely delicate task to compute cut points and cut loci
    although some low dimensional situations can be addressed (see, \eg,
    \cite{cutloci} where an approximation through averaging of the initial problem is
    studied). 
    Unfortunately, we rarely encounter the smooth case in practice, 
    and there is a lack of general
    method overcoming the different kind of singularities. 
    An extension of the smooth case method which uses the 
    Poincar\'e-Cartan integral invariant, see \cite{agrachev13},
     is easier to 
    generalize to non-smooth cases, and has been used to prove 
    local optimality for $L^1$ minimization
    of mechanical systems for instance, in \cite{caillau2016}. We
    use a similar technique to prove theorem 
    \ref{theoremsuf}, the main difference being the type of
    singularity: $L^1$-minimization of the control creates singularities of codimension
    one, and the extremal flow is
    the concatenation of the flows of two regular Hamiltonians. In
    our case, we have codimension two (and so unstable) singularities, and a Hamiltonian which fails to be 
    Lipschitzian.
    When the control lies in a box, second order conditions can be of use
    through a finite dimensional subsystem given by allowing the switching times to variate.
     Those techniques have been initiated by Stefani and Poggiolini, 
     see for instance \cite{gianna}.
      The majority of these works prove local optimality for normal
    extremal, and a few of them tackle the abnormal case. One can cite for 
    instance \cite{trelatabnormal} where single input systems are handled and
    can refer as well to \cite{numcontrolpol} where theoretical as well as numerical
    studies are leaded when the control lies in a polyhedron. We will also tackle
     only the normal case in the following, 
    since the co-dimension two singularity induced by minimizing the final
    time is our main focus. The recent paper \cite{biolosuf}
     from Agrachev and Biolo, proves local optimality of these broken extremal around
     the singularity with extra hypothesis on the adjoint state.
     Our approach is similar while in a slightly different framework (more suitable
     for mechanical systems) and easily checked by a simple numerical test.
     Thanks to this optimality analysis, we can investigate the regularity of a upper
     bound to the value function of this time optimal problem and prove
     that it is piecewise smooth. 

\section{The smooth theory} \label{sec:smooth}
\noi Let us begin by recalling the classical smooth case.
Consider an optimal control system on a manifold $M$:
\begin{equation}
\begin{cases}
\dot{x}=f(x,u),\; u\in U,\\
x(0)=x_0,\; x(t_f)=x_f,\\
C(u)=\int_0^{t_f}\varphi(x(t),u(t))dt\rightarrow\min
\end{cases}
\label{opt}
\end{equation} 
where $U\subset\mathbb{R}^m$, and $f:U\times M\rightarrow TM$ is a smooth family of vector field, and $\varphi:M\times U\rightarrow\mathbb{R}$ is the cost function. Define $H(x,p,p^0,u)=\langle p, f(x,u)\rangle+p^0\varphi(x,u)$ - $(x,p)\in T^*M$, $p^0$ a negative number, and u$\in U$ - the pseudo-Hamiltonian
associated with (\ref{opt}). By the classical Pontrjagin maximum principle
\cite{pontryagin1987mathematical}, optimal trajectories $x(t)$ associated with an
optimal control $u(t)$ are projections of the solutions $(x(t),p(t))$ of the
Hamiltonian system associated to $H$ such that, almost everywhere,
$H(x(t),p(t),u(t))=\underset{\tilde{u}\in U}{\max}H(x(t),p(t),\tilde{u})$. Solutions of this Hamiltonian system are called extremals.
We adapt a method presented in \cite{agrachev13}, \cite{caillau2016}
to deal with a codimension one singularity set.
Assume now that
\[ H^{\max}(x,p)=\underset{u\in U}{\max}H(x,p,u) \]
is $\C^2$-smooth and denote $\bar{z}(t)=(\bar{x}(t),\bar{p}(t))$, $t\in[0,\bar{t}_f]$, the
  extremal starting from $\bar{z}_0\in T^*M$. Let $\bar{u}$ be the
associated control, and 
consider the variational equation along $\bar{z}(t)$:
\begin{equation}
\dot{\delta z}=J\nabla^2 H(z(t))\delta z
\label{variationalsuf}
\end{equation}
where $J=\begin{pmatrix} 0 & -1\\
1 & 0
\end{pmatrix}.$ Solutions of (\ref{variationalsuf}) are called Jacobi fields.

\begin{definition}[Conjugate times]
A time $t_c$ is called a conjugate time if there exists a Jacobi field $\delta z$ 
such that 
$$d\pi(z(0))\delta z(0)=d\pi(z(t_c))\delta z(t_c)=0, $$
that is $\delta x(0)=\delta x(t_c)=0$, $\pi:T^*M \to M$ being the canonical projection.
We say that $\delta z$ is vertical at $0$ and $t_c$. The point $x(t_c)=\pi (z(t_c))$
 is a conjugate point.
\end{definition}

The following result implies optimality until the first conjugate time.
\begin{theorem}
	Assume that
	\begin{itemize}
		\item[(i)] The reference extremal $\bar{z}$ is normal,
		\item[(ii)] $\frac{\p x}{\p p_0}(t,\bar{z}_0)\neq0$ for all $t\in(0,t_f]$.
	\end{itemize}
	Then the reference trajectory $x$ is a local minimizer among all the
	$\C^0$-admissible trajectories with same endpoints.
	\label{sufsmooth}
\end{theorem}

\noi Assumption (ii) ensures disconjugacy along the reference extremal, and 
can be verified through a simple numerical test.
The proof consists in devising a Lagrangian manifold, and in propagating it using
the extremal flow.
One can then prove that the projection $\pi$ is invertible on a suitable submanifold: this allows to lift
admissible trajectories with same endpoints to the cotangent bundle, and to compare
them
using the Poincar\'e-Cartan invariant.
We will extend this proof to the non-smooth case encountered when minimizing time with
control on the disk.

\section{Setting} \label{sec:suff}
\noi Consider the following optimal time control system
\begin{equation}
\begin{cases}
\dot{x}=F_0(x)+u_1F_1(x)+u_2F_2(x), \\
t\in[0, t_f], \quad |u_1|^2+|u_2|^2 \leq 1,\\
x(0)=x_0,\quad x(t_f)=x_f,\\
t_f\rightarrow \min
\end{cases}
\label{Tmin3}
\end{equation}
so that the control set $U$ is the Euclidean disk and the fields $F_i$ are defined on a smooth four dimensional manifold $M$. We will use the following notation: $[F_i,F_j]:=F_{ij}$ for Lie brackets and $\{H_i,H_j\}:=H_{ij}$ for Poisson brackets. We denote 
$\mathcal{U}=L^{\infty}([0,t_f],U)$ the set of admissible controls, and 
make the following assumption: 
\begin{equation}
  \det(F_1(x), F_2(x), F_{01}(x), F_{02}(x))\neq 0,\quad x \in M.
\label{eq:A3}
\tag{A1}
\end{equation}
The (non-smooth) maximized Hamiltonian is
	\[H^{\max}(z)=H_0(z)+\sqrt{H_1^2(z)+H_2^2(z)}+p^0,\]
and the singular locus is $\Sigma:=\{H_1=H_2=0\}$.
One can make a comparison between those singularities and the 
double switchings obtained by taking $U=[-1,1]^2$ (or even $[-1,1]^m$). 
It has been
proved in \cite{controlbox} that extremals are optimal assuming some strong
Legendre-type conditions and the coerciveness of a second variation
to a finite dimensional problem obtained by perturbation of the switching times. 
If this result holds also for the abnormal case, our theorem does not require any coerciveness assumption.
The singular set (or \emph{swithching surface}) $\Sigma$ is partitioned into
three subsets as follows:
\begin{align*}
	\Sigma_- &=\{z \in T^*M\ |\ H^2_{12}(z)<H^2_{02}(z)+H^2_{01}(z)\},\\
	\Sigma_+ &=\{z \in T^*M\ |\ H^2_{12}(z)>H^2_{02}(z)+H^2_{01}(z)\},\\
	\Sigma_0 &=\{z \in T^*M\ |\ H^2_{12}(z)=H^2_{02}(z)+H^2_{01}(z)\}.
\end{align*}
According to \cite{orieuxcaillau18} and \cite{orieuxthese},
no regular extremal can reach $\Sigma_+$, so all extremals around this
 set are smooth, and Theorem~ \ref{sufsmooth} applies. The singular extremals
lying inside cannot be optimal via the Goh condition \cite{Goh}. 
According to Pontrjagin's Maximum principle, minimization of the final time implies
that normal extremals lie in the level sets $H=0$, for some $p^0\leq0$; normal extremals correspond to $p^0<0$ and abnormals to $p^0=0$.
We will deal with the $\Sigma_-$ case, which is the most relevant for applications,
notably because it contains mechanical systems.
We recall the result below from \cite{orieuxcaillau18}. (See also \cite{biolo-2017b}.)

\begin{theorem}
	In a neighbourhood $O_{\bar{z}}$ with $\bar{z}\in\Sigma_-$, existence and uniqueness
	of solution for the extremal flow hold, and all extremals are bang-bang,
	with at most 
	one switch. The extremal flow $z:(t,z_0)\in[0,t_f]\times O_{\bar{z}}\mapsto z(t,z_0)\in M$ is
	piecewise smooth. More precisely, there is a stratification
	$$O_{\bar{z}}=S_0\cup S^s\cup S^u\cup\Sigma$$ 
	where \begin{enumerate}
	    \item $S^s$ (resp. $S^u$) is the stable codimension-one submanifold of initial
	conditions leading to the switching surface (resp. in negative times),
	$S_0=O_{\bar{z}}\setminus (S^s\cup \Sigma)$;
	\item  the flow is smooth on $[0,t_f]\times S_0$, and 
	on $[0,t_f]\times S^s\setminus\Delta$  where $$\Delta:=\{(\bar{t}(z_0),z_0),\; z_0\in S^s\},$$
	and $\bar{t}(z_0)$ is the switching time of the extremal initializing at $z_0$;
	\item it is continuous on $O_{\bar{z}}$.
	\end{enumerate}
	\label{theorem}
\end{theorem}
\noi The set $S^s$ is the ensemble of initial conditions brought to the singular locus
by the flow,
$S^u$ is the set a initial conditions converging to $\Sigma$ in negative times.
In other words, the image of $S^s$ by the flow for times greater
than $\bar{t}(z_0)$.

\textbf{Example.}
A simple example of such a control system is given by nilpotent approximation of the minimum time Kepler (i.e., two-bodies) problem:

\begin{equation} \label{eq40}
\left\{\begin{array}{ll} \dot{x}_1=1+x_3 &
\dot{x}_3=u_1\\
\dot{x}_2=x_4 & \dot{x}_4=u_2\\\end{array} \right.
\end{equation}
with control on the $2$-disk, $u_1^2+u_2^2 \leq 1$. 

\noi Let $z_0\in S^s$.
\begin{prop}
The limits $\dot{z}(\bar{t}(z_0)_{\pm},z_0)$ as well as 
$\frac{\p z}{\p z_0}(\bar{t}(z_0),z_0)$ are well defined. 
\end{prop}

\begin{proof}
Both $\dot{z}(\bar{t}(z_0)_{\pm},z_0)$ are easily defined since the
control along an extremal has well defined right and left limits
at a switching time. 
Then, we already notice in \cite{orieuxcaillau18}, thanks to the normal form of proposition 2 that the map $z_0\mapsto \bar{z}(z_0)$ that associates to $z_0\in S^s$ the contact point with $\Sigma$ of the extremal initializing at $z_0$. Thus $\bar{z}(z_0)=z(\bar{t}(z_0),z_0)$ and the map $z_0\in S^s\mapsto z(\bar{t}(z_0),z_0)$ is smooth. This conclude the proof. 
\end{proof}
 
\noi For extremals outside $S^s$, the flow of the maximized Hamiltonian is smooth, 
and the usual sufficient conditions for optimality apply.
Let us denote $\bar{z}(t)$ our reference extremal, lying in $S^s$, with final time
 $\bar{t}_f$ and $\bar{t}:=\bar{t}(z_0)$, $\bar{z}(\bar{t}):=\bar{z}$. 
We assume that
the fiber $T^*_{\bar{x}_0}M$ and $S^s$ intersect transversally:
\begin{equation} \tag{A2}
  T^*_{\bar{x}_0}M\pitchfork S^s
\end{equation}
so $T^*_{\bar{x}_0}M\cap S^s$ is a smooth submanifold 
of dimension three.

\begin{definition}[exponential map]
	We call exponential mapping from $x_0$ the application
	\[ \exp_{\bar{x}_0}:(t,p_0)\in[0,t_f]\times T^*_{\bar{x}_0}M\cap S^s
	\rightarrow \pi(z(t,x_0,p_0)). \]
	\label{exp}
\end{definition}

\noi The exponential map is smooth except on $\Delta$, that is
when $x(t,x_0,p_0)\notin\Sigma$. 
The differential of the exponential mapping 
$d\exp_{x_0}(t,p_0)=(\dot{x}, \frac{\p x}{\p p_0})(t,p_0)$ is a $4\times4$ matrix, where
$\frac{\p}{\p p_0}$ denote the derivation with respect to a set of coordinates on 
$T^*_{x_0}M\cap S^s$. Set $M(t):=d\exp_{\bar{x}_0}(t,\bar{p}_0)$.

\begin{theorem}
Assume that \begin{itemize}
\item[(i)] The reference extremal is normal,
\item[(ii)] $\det M(t)\neq 0$ for all $t\in(0,\bar{t})\cup(\bar{t},\bar{t}_f)$ 
and $\det M(\bar{t}_-)\det M(\bar{t}_+)\neq0$,
\end{itemize}
then the reference trajectory is a minimizer among all $\C^0$ neighboring  trajectories
with same endpoints.
\label{theoremsuf}
\end{theorem}

\noi Obviously when $t=0$, $\frac{\p x}{\p p_0}(0,\bar{z}_0)=0$, 
and some part of the proof is dedicated to extend condition (ii) to
$t=0$.

\section{Proof of theorem \ref{theoremsuf}}
\label{sec:proofsuf}
\noi The rest of the paper is devoted to prove Theorem \ref{theoremsuf}.
 \begin{lemma}
Condition
(ii) implies that there exists a Lagrangian submanifold $\mathcal{L}$ transverse
to  $T^*_{x_0}M$, and close enough to $T^*_{x_0}M$ so 
$\mathscr{S}_0=\mathcal{L}\cap S^s$ is
a smooth submanifold of dimension 3, and such that 
$(\frac{\p x}{\p t},\frac{\p x}{\p z_0})(t,\bar{z}_0)$ is invertible on 
$[0,\bar{t})\times\mathscr{S}_0$, as well as on  $(\bar{t},t_f]\times\mathscr{S}_0$
($z_0$ denoting coordinates on $\mathscr{S}_0$). 
\label{lemlag}
\end{lemma}

\noi Moreover,
$(\frac{\p x}{\p t},\frac{\p x}{\p z_0})(\bar{t}_-,\bar{z}_0)$
and $(\frac{\p x}{\p t},\frac{\p x}{\p z_0})(\bar{t}_+,\bar{z}_0)$
are invertible.
Thus, the canonical projection $\pi$ is a diffeomorphism from 
$z((0,\bar{t})\times \mathscr{S}_0)$ onto its image
and an homeomorphism from
\begin{equation}
\mathscr{S}_1=\{z(t,z_0),\; (t,z_0)\in[0,\bar{t}(z_0)]\times\mathscr{S}_0\}
\end{equation}
onto its image. The same holds true for
\begin{equation}
\mathscr{S}_2=\{z(t,z_0),\; (t,z_0)\in[\bar{t}(z_0),t_f]\times\mathscr{S}_0\}.
\end{equation}
Let us prove that $\pi$ is a homeomorphism on their union. It is sufficient to prove
that the extremal crosses transversally
$\Sigma_1:=\Sigma\cap \mathscr{S}_1$. Since the map 
$(t,z_0)\in\R\times \mathscr{S}_0 \mapsto x(t,z_0)\in\pi(\mathscr{S}_1)$ is a 
homeomorphism, and is differentiable for all $(t,z_0)\neq(\bar{t}(z_0),z_0)$ with well
defined limits, we can define its inverse function $z_0(t,x)$, and 
$f(t,x)=t-\bar{t}(z_0(t,x))$. Thus we have $\Sigma_1=\{f=0\}$.
Now denote $g(t)=f(t,\bar{x}(t))$, we get
\begin{equation*}
\dot{g}(\bar{t}_-)=1=\dot{g}(\bar{t}_+)-d\bar{t}(\bar{z}_0)
\left[\frac{\p z_0}{\p t}+\frac{\p z_0}{\p x}\dot{\bar{x}}(t)\right].
\end{equation*}
Since \[ \frac{\p z_0}{\p t}(t,z_0(t,x)))=-\left(\frac{\p x}{\p
z_0}\right)^{-1}\dot{x}(t,z_0(t,x)), \]
we obtain $\dot{g}(\bar{t}_-)=\dot{g}(\bar{t}_+)=1$. In a neighbourhood of 
$\bar{z}$, every extremal passes transversely through $\pi(\Sigma_1)$: by restricting
 $\mathscr{S}_0$ if necessary, every extremal from  $\mathscr{S}_0$ passes transversely
  through $\pi(\Sigma_1)$, and the projection defines a continuous one to one mapping on 
$\mathscr{S}_1\cup\mathscr{S}_2$, and even a homeomorphism if we restrict ourselves to a compact neighbourhood of the reference extremal.

We will now prove that the Poincar\'e-Cartan form $\sigma=pdx-H^{\max}dt$ is exact on $\mathscr{S}_1$ and
$\mathscr{S}_2$. Let us first prove that $\sigma$ is closed on $\mathscr{S}_i$. 
Tangent vectors to $\mathscr{S}_1$ at $z$ are parametrized as follows:
$$\dot{z}(t,z_0)\delta t+\frac{\p z}{\p z_0}(t,z_0)\delta z_0,$$ with
 $(\delta t,\delta z_0)\in\R\times T_{z_0}\mathscr{S}_0,\; z(t,z_0)=z,$
whenever $z\notin\Sigma$. In that last case, tangent vectors
are given by 
$$ \dot{z}(t_-,z_0)\delta t+\frac{\p z}{\p z_0}(t_-,z_0)\delta z_0,$$
with $(\delta t,\delta z_0)\in\R\times T_{z_0}\mathscr{S}_0,\; z(t,z_0)=z.$
Let $(v_1,v_2)\in T\mathscr{S}_i$, we have
\begin{equation*}
\begin{split}
d\sigma(v_1,v_2)=dp\wedge dx(\frac{\p z}{\p z_0}(t,z_0)\delta z_0^1,
\frac{\p z}{\p z_0}(t,z_0)\delta z_0^2)\\
-dH\wedge dt(v_1,v_2)=\omega (\delta z_0^1,\delta z_0^2)
\end{split}
\end{equation*}
because the flow is symplectic on $S^s$, and $dH.\dot{z}=0$. Eventually,
$\omega (\delta z_0^1,\delta z_0^2)=0$ since $\mathscr{S}_0\subset\mathcal{L}$
 is isotropic. This equality still holds for 
tangent vectors at $(\bar{t}(z_0),z(\bar{t}(z_0),z_0))$.
Being closed, the Poincar\'e form is actually exact on each $\mathscr{S}_i$. Indeed, consider a
curve $\gamma(s)=(t(s),z(t(s),z_0(s)))$ on $\mathscr{S}_1\cup\mathscr{S}_2$: it retracts 
continuously on $\gamma_0(s)=(0,z_0(s))$. Then, since $\sigma$ is closed, 
$$\int_\gamma\sigma=\int_{\gamma_0}\sigma,$$ and one can chose
 $\mathcal{L}$ as the graph of the differential of a smooth function, so
$$\int_{\gamma_0}\sigma=0$$
by Stokes formula.
Let us finally prove that our reference extremal $t\in[0,\bar{t}_f]\mapsto\bar{z}(t)=(\bar{x}(t),\bar{p}(t))$  
minimizes the final time among all close $\C^1$-curves with same endpoints.
Let $x(t)$, $t\in[0,t_f]$ be a $\C^1$ admissible curve, generated
 by a control $u$ with $x(0)=x_0$, $\C^0$ close to $\bar{x}$, then, denote $z(t)=(x(t),p(t))$
  its well defined lift in $\mathscr{S}_1\cup\mathscr{S}_2$. Then
  $$\int_0^{t_f}p^0=\int_0^{t_f}p.\dot{x}-H(x,p,u)dt\geq
  \int_0^{t_f}p\dot{x}-H^{\max}(x,p)dt.$$
Since $\sigma$ is exact, the right-hand side is actually
$$\int_z \sigma=\int_{\bar{z}}\sigma$$
Thus, 
\[t_fp^0\leq\int_{\bar{z}}\sigma=\bar{t}_fp^0, \]
which proves local optimality for the reference trajectories in the normal case, 
among all $\C^0 $-close curves that have
$\C^1$ regularity. A perturbation argument allows us to conclude on optimality with
respect to all continuous admissible curves; which ends the proof of Theorem~\ref{theoremsuf}.\\

In the very specific case when $T^*_{x_0}M\subset S^s$, one has to change a
 bit the exponential mapping defined above, but the same proof
basically holds.

\begin{proof}[Proof of Lemma \ref{lemlag}] We follow and adapt the proof in
\cite{caillau2016} appendix.
Let $S_0$ be a symmetric matrix
so that the Lagrangian subspace $L_0=\{\delta x_0=S_0\delta p_0 \}$ 
intersects transversely $T_{\bar{z}_0}S^s$. Consider the two linear symplectic
systems
\[\delta\dot{z}(t)=\frac{\p H^{\max}}{\p z}(\bar{z}(t))\delta z(t)\]
$t\in[0,\bar{t}[, \delta z(0)=(S_0,I) $
and 
\[\dot{\phi}(t)=\frac{\p H}{\p z}(\bar{z}(t),u(t))\phi(t), \; t\in[0,\bar{t}[, \phi(0)=I.\]
Set $\delta \tilde{z}(t)=(\delta\tilde{x}(t),\delta\tilde{p}(t))=\phi(t)^{-1}\delta z(t)$. 
Since $\delta\tilde{z}(0)=\delta z(0)=(S_0,I)$, the matrix $$S(t)=\delta \tilde{x}(t)\delta\tilde{p}(t)^{-1}$$
exists for small enough $t$. It is symmetric since 
$$L_t=\exp(X_{H^{\max}}t)'(L_0)\text{ and } (\phi(t))^{-1}(L_t)$$
are Lagrangian submanifolds.
One can prove that $\dot{S}(t)\geq0$  (see \cite{caillau2016}, annex),  whenever $S(t)$ is defined, as the consequence of the classical first
and second order conditions on the maximized Hamiltonian.
Then, if $S_0>0$ (small enough so that $S(t)$ is defined on $[0,\epsilon]$), $S(t)$ is invertible, and as such, 
$\phi(t)^{-1}(L_t)\pitchfork \ker d\pi(\bar{z}_0)$. 
This implies $L_t\pitchfork\ker d\pi(\bar{z}(t))$ since
$\phi(t)(\ker d\pi(\bar{z}_0))=\ker d\pi(\bar{z}(t))$.
There exists a Lagrangian submanifold $\mathcal{L}_0$ of $T^*M$ 
tangent at $L_0$ in $\bar{z}_0$. It intersects $S^s$ transversely,
and the lemma follows.
\end{proof}

\section{Regularity of the field of extremals} \label{sec:valuefunc}
\noi Fix $x_0\in M$, the value function associates to a final state 
the optimal cost, and is defined as
$$S_{x_0}:x_f\in M\mapsto \inf_{u\in\mathcal{U}}\{t_f,\; x(t_f,u)=x_f \}\in\mathbb{R}.$$
It defines a pseudo-distance between $x_0$ and $x_f$ and
its regularity is a crucial information in optimal control problem,
especially in sub-Riemannian geometry where it defines the
distance. We give the regularity of the final time for extremals
that are locally optimal under the assumptions of the previous section.
If they are globally optimal (which holds true for small enough
times), this final time coincides with the value function while, otherwise, we
only obtain the regularity of an upper bound to the value function.
Actually, since the differential equation is homogeneous in 
the adjoint vector, one can restrict to the unitary 
bundle of the cotangent bundle $ST^*M$, and consider 
$$\exp:(t_f,p_0)	\in \mathbb{R}_+\times ST^*_{x_0}(M)\mapsto x(t_f,x_0,p_0).$$
The authors have shown in \cite{orieuxcaillau18} that this function is piecewise 
smooth, and belongs to the log-exp category. There are two cases:

\paragraph{First case} In the neighbourhood of 
$(x_0,\bar{p}_0)\notin S^s$, the extremal flow, as well as $F(t_f,p_0):=\exp(t_f,p_0)-x_f$, 
are smooth. If $dF(\bar{t}_f,\bar{p}_0)$ is invertible, that is
if $\det(\dot{x}(t,x_0,\bar{p}_0),\frac{\p x}{\p p_0}
(t,x_0,\bar{p}_0))\neq 0$, for all $t$, where $p_0$ is a system of 
coordinates on $ST^*(M)$ around $(x_0,\bar{p}_0)$ then, locally, 
we have a $\C^1$ inverse $F^{-1}(x_f)=(t_f,p_0)(x_f)$.
This is the well-known smooth case.

\paragraph{Second case} In the neighbourhood $(x_0,\bar{p_0})\in S^s$,
then replace the previous definition of exponential mapping by 
$$\exp:(t_f,p_0)\in \mathbb{R}_+\times S^s\cap T_{x_0}^*M\mapsto x(t_f,x_0,p_0).$$
Under the transversality condition, $S^s\cap T_{x_0}^*M$ is a smooth 
\noi $3$-dimensional submanifold, and since the flow is smooth on $S^s$, the same process 
can be applied with the same result, except when $x_f\in \Sigma$.
In such a neighbourhood, we only have PC$^1$ regularity
and we need a weaker inverse function theorem. 
We use a result from \cite{implPC1}.
 
\begin{theorem}
Assume
\begin{itemize}
\item[(i)] $\det(\dot{x}(t,\bar{z}_0),\frac{\p x}{\p p_0}
(\bar{t}_f,\bar{z}_0))\neq 0$ for all $t\neq\bar{t}$
\item[(ii')] the two determinants
\[ \det(\dot{x}(\bar{t}_-,\bar{z}_0),\frac{\p x}{\p p_0}
  (\bar{t}_{f-},\bar{z}_0)),\ 
  \det(\dot{x}(\bar{t}_+,\bar{z}_0),\frac{\p x}{\p p_0}
  (\bar{t}_{f+},\bar{z}_0)) \]
have the same sign.
\end{itemize} 
Then the final time $x_f\mapsto t_f(x_f)$, is continuous and piecewise $\C^1$ in a neighbourhood of $x(\bar{t}_f,x_0,\bar{p}_0)$.
\end{theorem}

\begin{proof}
Thanks to (i) and (ii') we have a PC$^1$ inverse, by Theorem~3 in
\cite{implPC1}, so $x_f\mapsto(t_f(x_f),p_0(x_f))$ is
piecewise continuously differentiable.
\end{proof}
 
\noi Obviously (ii') implies (ii) in Theorem~\ref{theoremsuf}, and the extremal is
locally optimal.
When it is globally optimal, the value function is
$S(x_f)=t_f(x_f)$, the final time of the extremal.
Otherwise, $S(x_f)\leq t_f$ 
and we only have PC$^1$ regularity for an upper bound function to the value function.

\section{Conclusions}
We managed to go beyond the smooth case, and tackle the case of optimization of the final time, inducing a Hamiltonian admitting singularities, for a large class of control-affine, that includes mechanical systems.  Though the Hamiltonian itself has just Lipschitz regularity, we show that the necessary tools to prove optimality conditions as well as a disconjugacy hypothesis can still be defined. An interesting development of this work would involve simulations and study of the behavior of switching and conjugate times for simple problem from space mechanics, namely, a two or restricted circular three bodies problem.


\begin{thebibliography}{10}

\bibitem{gianna}
A.~A.~Agrachev, G.~Stefani and P.~Zezza.
\newblock Strong optimality for a bang-bang trajectory.
\newblock {\em SIAM J. Control Optim.}, 41:991--1014, 2002.

\bibitem{biolo-2017b}
Agrachev, A.~A.; Biolo, C.
Switching in time-optimal problem: The 3D case with 2D control.
{\em J.\ Dyn.\ Control Syst.}\ \textbf{23}, no.~3, 577--595, 2017.

\bibitem{biolosuf}
A.~A. Agrachev and C.~Biolo.
\newblock Optimality of broken extremals.
\newblock {\em J. Dynamical and Control Systems}, to appear.

\bibitem{agrachev13}
A.~A. Agrachev and Y.~L. Sachkov.
\newblock {\em Control theory from the geometric viewpoint}, volume~87 of {\em
  Encyclopaedia of Mathematical Sciences}.
\newblock Springer-Verlag, Berlin, 2004.
\newblock Control Theory and Optimization, II.

\bibitem{cutloci}
B.~Bonnard, J.-B.~Caillau and G.~Janin.
\newblock Conjugate and cut loci of a two-sphere of revolution with application
  to optimal control.
\newblock {\em Ann. Inst. Henri Poincar\'e Anal. Non Lin\'eaire},
  26:1081--1098, 2009.

\bibitem{caillau2016}
J.-B.~Caillau, Z.~Chen, and Y.~Chitour.
\newblock ${L}^1$-minimization for mechanical systems.
\newblock {\em SIAM J. Control Optim.}, 54:1245--1265, 2016.

\bibitem{orieuxcaillau18}
J.-B.~Caillau, J.~F\'ejoz, M.~Orieux, and R.~Roussarie.
\newblock Singularities of min. time affine control systems.
\newblock {\em Submitted}, 2018.

\bibitem{Filipov}
L.~Cesari.
\newblock {\em Optimization - Theory and Applications}.
\newblock Problems with Ordinary Differential Equations. Springer, 1983.

\bibitem{Goh}
B.~S.~Goh.
\newblock Necessary conditions for singular extremals involving multiple
  control variables.
\newblock {\em SIAM J. Control}, 4:716--731, 1966.

\bibitem{implPC1}
M.~Gowda.
\newblock Inverse and implicit function theorems for h-differentiable and
  semismooth functions.
\newblock {\em Optimization and Software}, 19:443--461, 2004.

\bibitem{numcontrolpol}
H.~Maurer and N.~Osmolovskii.
\newblock Second order sufficient conditions for time optimal bang-bang
  control.
\newblock {\em SIAM J. Control Optim.}, 42:239--2263, 2004.

\bibitem{orieuxthese}
M.~Orieux.
\newblock Some properties and applications of minimum time control.
\newblock {\em PhD. Thesis, hal.inria.fr/tel-01956833v1}, 2018.

\newpage

\bibitem{controlbox}
L.~Poggiolini and M.~Spadini.
\newblock Bang-bang trajectories with a double switching time in the minimum
  time problem.
\newblock {\em ENSAIM: COCV}, 22:688 -- 709, 2016.

\bibitem{pontryagin1987mathematical}
L.~S. Pontrjagin.
\newblock The mathematical theory of optimal processes and differential games.
\newblock {\em Trudy Mat. Inst. Steklov.}, 169:119--158, 254--255, 1985.
\newblock Topology, ordinary differential equations, dynamical systems.

\bibitem{trelatabnormal}
E.~Tr\'elat.
\newblock Asymptotics of accessibility sets along an abnormal trajectory.
\newblock {\em ENSAIM: COCV}, 6:387--414, 2001.

\end{thebibliography}
\end{document}